\documentclass[12pt]{amsart}
\usepackage[foot]{amsaddr}
\usepackage{lipsum}
\usepackage{amssymb}
\usepackage{enumerate}
\usepackage{amsthm}
\usepackage{tikz-cd}
\usepackage{amsmath}
\usepackage[mathscr]{euscript}
\usepackage[utf8]{inputenc}
\usepackage[english]{babel}
\usepackage{hyperref}
\hypersetup{
    colorlinks=true,
    linkcolor=blue,
    citecolor=red,
    filecolor=red,      
    urlcolor=violet,
}

\makeatletter
\@namedef{subjclassname@2020}{%
  \textup{2020} Mathematics Subject Classification}
\makeatother
\newtheorem{thm}{Theorem}[section]
\newtheorem{lem}[thm]{Lemma}

\newtheorem{defi}[thm]{Definition}

\newtheorem{corl}[thm]{Corollary}
\newtheorem{xrem}{Remark}
\newtheorem{exm}[thm]{Example}

\DeclareMathOperator{\Nef}{{Nef}}

\DeclareMathOperator{\rank}{{rank}}
\DeclareMathOperator{\ME}{{ME}}

\DeclareMathOperator{\Pic}{{Pic}}
\DeclareMathOperator{\Sym}{{Sym}}
\DeclareMathOperator{\Eff}{{Eff}}
\DeclareMathOperator{\End}{{End}}
\DeclareMathOperator{\Div}{{Div}}
\DeclareMathOperator{\Proj}{{Proj}}
\DeclareMathOperator{\Bl}{{Bl}}

\begin{document}
\baselineskip=16pt

\author{Snehajit Misra}

\address{\rm Author's Address : Tata Institute of Fundamental Research (TIFR), Homi Bhabha Road, Mumbai 400005, India.}
\email{smisra@math.tifr.res.in}

\begin{abstract}
In this article, we consider the projective bundle $\mathbb{P}_X(E)$ over a smooth complex projective variety $X$, where $E$ is a semistable bundle on
$X$ with $c_2(\End(E)) =0$. We give a necessary and sufficient condition to get the equality 
$ \Nef^1\bigl(\mathbb{P}_X(E)\bigr) = \overline{\Eff}^1\bigl(\mathbb{P}_X(E)\bigr)$ of nef cone and pseudo-effective cone of divisors in 
$\mathbb{P}_X(E)$. As an application of our result, we show the equality of nef and pseudo-effective cones of divisors of projective bundles over some 
special varieties. In particular, we show that weak Zariski decomposition exists on these projective bundles. We also show that weak Zariski
decomposition exists for fibre product $\mathbb{P}_C(E)\times_C\mathbb{P}_C(E')$ over a smooth projective curve $C$. Finally, we show that a semistable
bundle $E$ of rank $r\geq 2$ with $c_2\bigl(\End(E)\bigr) = 0$ on a smooth complex projective surface of Picard number 1 is $k$-homogeneous i.e. 
$\overline{\Eff}^k\bigl(\mathbb{P}_X(E)\bigr) = \Nef^k\bigl(\mathbb{P}_{X}(E)\bigr)$ for all $1 \leq k < r$.
\end{abstract}

\title{Pseudoeffective cones of projective bundles and weak Zariski decomposition}

\maketitle

\section{introduction}
 In the last few decades, a number of notions of positivity have been used to understand the geometry of the higher dimensional projective varieties. 
 The cone of nef $\mathbb{R}$-divisors on a smooth projective variety $X$, denoted by $\Nef^1(X)$, and closure of the cone generated by effective
 $\mathbb{R}$-divisors on $X$ called as pseudo-effective cone, denoted by $\overline{\Eff}^1(X)$, are two fundamental invariants of $X$ which play a
 very crucial role in this understanding. 
 
Most of the developments in this direction has been successfully summarized in \cite{L1}, \cite{L2}.
In \cite{Z62}, Zariski proved that for any effective divisor $D$ on a smooth projective surface $X$, there exists uniquely determined 
$\mathbb{Q}$-divisors $P$ and $N$ with  $D = P + N$ satisfying the following:

(i) $P$ is nef and $N$ is effective.

(ii) either $N$ is zero or has negative definite intersection matrix of its components.

(iii) $P\cdot C = 0$ for every irreducible component $C$ of $N$.

 Fujita \cite{F79} later extended the above decomposition for pseudo-effective $\mathbb{R}$-divisors on surfaces. This decomposition gives a great deal of information about the linear series on a smooth surface. For example, such decomposition is used as a very powerful tool to prove the rationality of the volume of integral divisors on surfaces. Many attempts has been made to generalize this decomposition on higher dimensional varieties, for instance the Fujita-Zariski decomposition \cite{F86} and the CKM-Zariski decomposition \cite{P03} . In fact, the rationality of the volume of a big divisor admitting CKM-Zariski decomposition is proved similarly. However, Cutkosky \cite{C86} constructed examples of effective big divisors on higher dimensional varieties  with irrational volume, which proved that a CKM-Zariski decomposition on a smooth projective variety does not exist in general. Nakayama \cite{N04} constructed a related example to show that it is impossible to find a CKM-Zariski decomposition 
even if one allows $N$ and $P$ to be $\mathbb{R}$-divisors in the definition of CKM-Zariski decomposition.
\begin{defi}
\rm (\it Weak Zariski Decomposition \rm) A weak Zariski decomposition for an $\mathbb{R}$-divisor
$D$ on a normal projective variety $X$ consists of a normal variety $X'$, a projective bi-rational morphism $f : X' \longrightarrow X$, and a numerical equivalence $f^*D \equiv  P + N$ such that $P \in \Nef^1(X')$ and $N \in {\Eff}^1(X')$.
\end{defi}
The existence of weak Zariski decomposition also has very useful consequences. For example, in \cite{B12}, it is shown that the existence of weak Zariski decomposition for adjoint divisors $K_X + B$ of a log canonical pair $(X,B)$ is equivalent to the existence of log minimal model for $(X,B)$. However, in \cite{Le14} an example is found of a pseudo-effective $\mathbb{R}$-divisor on the blow-up of $\mathbb{P}^3$ in 9 general points, which does not admit a weak Zariski decomposition.

The existence of weak Zariski decomposition of pseudo-effective $\mathbb{R}$-divisors on projective bundles $\mathbb{P}_X(E)$ has been studied in \cite{N04} and in \cite{M-S-C}, when the base space $X$ is a smooth curve and a smooth projective variety with Picard number 1 respectively. In this paper, we consider a vector bundle $E$ of rank $r$ on a smooth projective variety $X$ together with the projectivization map $\pi : \mathbb{P}_X(E) \longrightarrow X$. We investigate the existence of weak Zariski decomposition for pseudo-effective $\mathbb{R}$-divisors on $\mathbb{P}_X(E)$. By \cite{Le14}, the numerical classes of $\mathbb{R}$-divisors in a smooth projective variety $X$ admitting weak Zariski decomposition forms a cone in $\overline{\Eff}^1(X)$, and hence it is enough to prove existence of such decomposition for every extremal ray of the pseudo-effective cone $\overline{\Eff}^1(X)$ (see Proposition 1 in \cite{M-S-C}). In particular, if ${\Eff}^1(X)$ is a closed cone ( e.g., when $X$ is a Mori dream space), then weak Zariski decomposition always exists. Also,   if $\Nef^1\bigl(\mathbb{P}_X(E)\bigr) = \overline{\Eff}^1\bigl(\mathbb{P}_X(E)\bigr)$ for some vector bundle $E$ on a specific $X$, then weak Zariski decomposition always exist on $\mathbb{P}_X(E)$ because all the extremal rays are nef generators in this case. However, pseudo-effectivity does not imply nefness in $\mathbb{P}_X(E)$ in general.

In his paper \cite{Mi}, Miyaoka found that in characteristic 0,  
a vector bundle $E$ on a smooth projective curve $C$ is semistable if and only if the
nef cone $\Nef^1\bigl(\mathbb{P}_C(E)\bigr)$ and pseudo-effective cone $\overline{\Eff}^1\bigl(\mathbb{P}_C(E)\bigr)$ coincide.
However, examples are given in \cite{B-H-P} to show that this characterization of semistability in terms of equality of nef and pseudo-effective cone 
is not true for vector bundles on higher dimensional projective varieties. More specifically, they produce examples of semistable bundles $E$ on any 
smooth projective variety $X$ with $c_2\bigl(\End(E)\bigr)$ not numerically zero, but
$\Nef^1\bigl(\mathbb{P}_X(E)\bigr) \subsetneq \overline{\Eff}^1\bigl(\mathbb{P}_X(E)\bigr)$ (see section 5, \cite{B-H-P}). 
In section 3 of this article, we prove the following :
\begin{thm}\label{thm1.2}
Let $E$ be a semistable vector bundle of rank $r$ on a smooth complex projective variety $X$ with $c_2\bigl(\End(E)\bigr) = 0$, and 
 $\pi : \mathbb{P}_X(E) \longrightarrow X$ be the projectivization map. Denote 
 $\lambda_E := c_1(\mathcal{O}_{\mathbb{P}_X(E)}(1)) - \frac{1}{r}\pi^*c_1(E)$. Then the following are equivalent:
 \begin{itemize}
  \item{$(1)$ \it $\Nef^1(X) = \overline{\Eff}^1(X)$.}
  \item{$(2)$ \it $c_1\bigl(\pi_*\mathcal{O}_{\mathbb{P}_X(E)}(D)\bigr) \in \Nef^1(X)$ for every effective divisor $D$ in $\mathbb{P}_X(E)$.}
  \item{$(3)$ \it $\Nef^1\bigl(\mathbb{P}_X(E)\bigr) = \overline{\Eff}^1\bigl(\mathbb{P}_X(E)\bigr).$}
 \end{itemize}
 In addition, if \hspace{1mm} $\Nef^1(X) = \overline{\Eff}^1(X)$ is a finite polyhedron generated by nef classes $L_1, L_2, \cdots,L_k$, then
 \begin{align*}
  \Nef^1\bigl(\mathbb{P}_X(E)\bigr) = \overline{\Eff}^1\bigl(\mathbb{P}_X(E)\bigr) = 
  \Bigl\{y_0\lambda_E + y_1\pi^*L_1+\cdots+y_k\pi^*L_k\mid y_i \in \mathbb{R}_{\geq 0}\Bigr\},
 \end{align*}
\end{thm}
In general, the nef cones $\Nef^1\bigl(\mathbb{P}_X(E)\bigr)$ might not be a finite polyhedron when the Picard number $\rho(X)$ is at least 3, and in 
such cases these cones are not so easy to calculate. As applications of the above result, we show the equality $\Nef^1\bigl(\mathbb{P}_X(E)\bigr) = \overline{\Eff}^1\bigl(\mathbb{P}_X(E)\bigr)$ for semistable bundles $E$ with $c_2(\End(E)) = 0 $ on different smooth projective varieties $X$ with high Picard numbers (see Corollary \ref{corl3.5}, Corollary \ref{corl3.8}). For instance, we prove the following:
\begin{corl}
$($Corollary \ref{thm3.9}$)$ Let $X$ be a smooth complex projective variety $X$ with $\overline{\Eff}^1(X) = \Nef^1(X)$ and $E_1,E_2,\cdots,E_k$ be finitely many semistable vector bundles on $X$ of ranks $r_1,r_2,\cdots,r_k$ respectively with $c_2\bigl(\End(E_i)\bigr) = 0$ for all $i \in \{1,2,\cdots,k\}$. Then
  \begin{center}
 $\Nef^1\bigl(\mathbb{P}_X(E_1) \times_X \mathbb{P}_X(E_2)\times_X\cdots\times_X \mathbb{P}_X(E_k)\bigr) = \overline{\Eff}^1\bigl(\mathbb{P}_X(E_1) \times_X \mathbb{P}_X(E_2)\times_X \cdots\times_X \mathbb{P}_X(E_k)\bigr).$
 \end{center}
\end{corl}
We also give several examples to illustrate Theorem \ref{thm1.2}. In all these cases, weak Zariski decomposition always exists for pseudo-effective divisors as mentioned in our previous remark. In section 4, we prove the following :
\begin{thm}\label{thm1.4}
$($Theorem \ref{thm5.4}$)$ Let $E$ and $E'$ be two vector bundles of rank $m$ and $n$ respectively on a smooth complex projective curve $C$. Then weak Zariski decomposition exists for any pseudo-effective $\mathbb{R}$-divisors on the fibre product $\mathbb{P}_C(E)\times_C\mathbb{P}_C(E')$.
\end{thm}
Note that $\mathbb{P}_C(E)\times_C\mathbb{P}_C(E') \simeq \mathbb{P}_{\mathbb{P}(E')}(\pi_1^*E)$ where $\pi_1 : \mathbb{P}_C(E') \longrightarrow C$ is the projection map.  Here $\mathbb{P}_C(E')$ has Picard number 2 and $\Nef^1\bigl(\mathbb{P}_C(E)\times_C\mathbb{P}_C(E')\bigr)  \neq \overline{\Eff}^1\bigl(\mathbb{P}_C(E)\times_C\mathbb{P}_C(E')\bigr)$ unless both $E$ and $E'$ are semistable bundles on $C$ (see Corollary \ref{thm3.9}). Also, $\mathbb{P}_C(E)\times_C\mathbb{P}_C(E')$ is not a Mori dream space unless $C = \mathbb{P}^1$. In case $C = \mathbb{P}^1$, the fibre product $\mathbb{P}_C(E)\times_C\mathbb{P}_C(E')$ gives an example of a Bott tower of height 3. The proof of Theorem \ref{thm1.4} make use of some of the techniques in \cite{M-S-C}.

Like the cones of divisors, the closure of cones generated by effective $k$-cycles $\overline{\Eff}_k(X)$ and its dual $\Nef^k(X)$ are also of great importance for any projective variety $X$. It is not very easy to calculate these cones of higher co-dimensional cycles in higher dimensional varieties (see \cite{D-E-L-V}, \cite{Fu}, \cite{C-C}, \cite{C-L-O}). Let $E$ be an elliptic curve and $B = E^{\otimes n}$ be the $n$-fold product. Then $B$ is an abelian variety. In \cite{D-E-L-V}, it is shown that for any $ 2 \leq k \leq n$, there exists nef cycles of codimension $k$ in $\Nef^k(B)$  that are not pseudo-effective. Hence, the inclusion $\Nef^k(X) \subseteq \overline{\Eff}^k(X)$ for higher codimensional cycles $(2 \leq k \leq n)$ does not hold true in general. However, in \cite{Fu}, it is proved that a vector bundle $E$ of rank $r$ on a smooth complex projective curve $C$ is semistable if and only if it is $k$-homogeneous i.e.  $\overline{\Eff}^k\bigl(\mathbb{P}_C(E)\bigr) = \Nef^k\bigl(\mathbb{P}_{C}(E)\bigr)$ for all $k \in \{1,2,\cdots,r-1\}$. But there is an example of a rank 2 stable bundle $E$ on $\mathbb{P}^2$ (see Example 3.4 in \cite{Fu}) which is not $1$-homogeneous i.e. $\overline{\Eff}^1\bigl(\mathbb{P}_{\mathbb{P}^2}(E)\bigr) \neq \Nef^1\bigl(\mathbb{P}_{\mathbb{P}^2}(E)\bigr)$. We note that any semistable bundle $E$ on $\mathbb{P}^2$ with $c_2\bigl(\End(E)\bigr) = 0$ is trivial bundle up to some twist by a line bundle, and hence can never be stable.
Motivated by this observation, we prove the following in section 5:
\begin{thm}\label{thm1.5}
$($Theorem \ref{thm4.2}$)$ Let $E$ be a semistable vector bundle of rank $r \geq 2$ with $c_2\bigl(\End(E)\bigr) = 0$ on a smooth complex projective surface $X$ of Picard number $\rho(X) = 1$. We denote the numerical class of a fibre of the projection map $\pi : \mathbb{P}_X(E) \longrightarrow X$ by $F$. If $L_X \in N^1(X)_{\mathbb{R}}$ denotes the numerical class  of an ample generator of the N\'{e}ron-Severi group $N^1(X)_{\mathbb{Z}}$, then 
\begin{center}
$\overline{\Eff}^k\bigl(\mathbb{P}_X(E)\bigr) = \Bigl\{ a \lambda_E^k + b \lambda_E^{k-1}\pi^*L_X + c\lambda_E^{k-2}F \mid a,b,c \in \mathbb{R}_{\geq 0} \Bigr\}$ for all $k$ satisfying $ 1 < k < r$.
\end{center}
Moreover, $E$ is a $k$-homogeneous bundle on $X$ i.e., $\overline{\Eff}^k\bigl(\mathbb{P}_X(E)\bigr) = \Nef^k\bigl(\mathbb{P}_X(E)\bigr)$ for all $1 \leq k < r$.
\end{thm}
We also obtain similar results for $k$-homogeneous bundles on some smooth ruled surfaces (see Theorem \ref{thm4.3}).
Note that ruled surfaces have Picard number 2.
\subsection*{Acknowledgement}
The author would like to thank Indranil Biswas and Omprokash Das for many useful discussions. This work is supported financially by a postdoctoral fellowship from TIFR, Mumbai under DAE, Government of India.
\section{Preliminaries}
 \subsection{Nef cone and Pseudo-effective cone}
 Throughout this article, all the algebraic varieties are assumed to be irreducible.
 Let $X$ be a projective variety of dimension $n$ and $\mathcal{Z}_k(X)$ (respectively $\mathcal{Z}^k(X)$) denotes the free abelian group generated by $k$-dimensional (respectively $k$-codimensional) subvarieties on $X$. The Chow groups $A_k(X)$ are defined as the quotient of $\mathcal{Z}_k(X)$ modulo rational equivalence. When $X$ is a  smooth irreducible projective variety, we denote $A^k(X) := A_{n-k}(X)$. 
  
  Two cycles $Z_1$,$Z_2 \in \mathcal{Z}^k(X)$ are said to be numerically equivalent, denoted by $Z_1\equiv Z_2$ if $Z_1\cdot \gamma =  Z_2\cdot\gamma $ 
  for all $\gamma \in \mathcal{Z}_k(X)$. The \it numerical groups \rm  $ N^k(X)_{\mathbb{R}}$ are defined as the quotient of 
  $\mathcal{Z}^k(X) \otimes \mathbb{R}$ modulo numerical equivalence. When $X$ is smooth, we define 
  $N_k(X)_{\mathbb{R}} := N^{n-k}(X)_{\mathbb{R}}$ for all $k$. 
  
  Let 
  \begin{center}
  $\Div^0(X) := \bigl\{ D \in \Div(X) \mid D\cdot C = 0 $ for all curves $C$ in $X \bigr\} \subseteq \Div(X)$.
  \end{center}
 be the subgroup of $\Div(X)$ consisting of numerically trivial divisors. The quotient $\Div(X)/\Div^0(X)$ is called the N\'{e}ron Severi group of $X$, and is denoted by $N^1(X)_{\mathbb{Z}}$.
   The N\'{e}ron Severi group  $N^1(X)_{\mathbb{Z}}$ is a free abelian group of finite rank.
 Its rank, denoted by $\rho(X)$ is called the Picard number of $X$. In particular, $N^1(X)_{\mathbb{R}}$ is called the real N\'{e}ron 
 Severi group and $N^1(X)_{\mathbb{R}}  := N^1(X)_{\mathbb{Z}} \otimes \mathbb{R} := \bigl(\Div(X)/\Div^0(X)\bigr) \otimes \mathbb{R}$.
 For $X$ smooth, the intersection product induces a perfect pairing 
 \begin{align*}
N^k(X)_{\mathbb{R}} \times N_k(X)_{\mathbb{R}} \longrightarrow \mathbb{R}
\end{align*}
which implies $N^k(X)_{\mathbb{R}} \simeq (N_k(X)_{\mathbb{R}})^\vee$ for every $k$ satisfying $0\leq k \leq n$. The direct sum 
  $N(X) := \bigoplus\limits_{k=0}^{n}N^k(X)_{\mathbb{R}}$ is a graded $\mathbb{R}$-algebra with multiplication induced by the intersection form.
 
The convex cone generated by the set of all effective $k$-cycles in $N_k(X)_\mathbb{R}$ is denoted by $\Eff_k(X)$ and its closure $\overline{\Eff}_k(X)$ is called the \it pseudo-effective cone \rm  of $k$-cycles in $X$. For any $0 \leq k \leq n$,
$\overline{\Eff}^k(X) := \overline{\Eff}_{n-k}(X)$. The \it nef cone \rm  are defined as follows :
\begin{align*}
 \Nef^k(X) := \bigl\{ \alpha \in N^k(X) \mid \alpha \cdot \beta \geq 0 \hspace{2mm} \forall \beta \in \overline{\Eff}_k(X)\bigr\}.  
\end{align*}
An irreducible curve $C$ in $X$ is called \it movable \rm if there exists an algebraic family of irreducible curves $\{C_t\}_{t\in T}$ such 
that $C = C_{t_0}$ for some $t_0 \in T$ and $\bigcup_{t \in T} C_t \subset X$ is dense in $X$.

A class $\gamma \in N_1(X)_{\mathbb{R}}$ is called movable if there exists a movable curve $C$
such that $\gamma = [C]$ in $N_1(X)_{\mathbb{R}}$. The closure of the cone generated by movable classes in
$N_1(X)_{\mathbb{R}}$, denoted by $\overline{\ME}(X)$ is called the \it movable cone\rm. By \cite{BDPP13} $\overline{\ME}(X)$ is 
the dual cone to $\overline{\Eff}^1(X)$. Also, $\Nef^1(X) \subseteq \overline{\Eff}^1(X)$ always (see Ch 2 \cite{L1}). We refer the 
reader to \cite{L1},\cite{L2} for more details about these cones.
\subsection{Semistability of Vector bundles}
Let $X$ be a smooth complex projective variety of dimension $n$ with a fixed ample line bundle $H$ on it.
For a torsion-free coherent sheaf $\mathcal{G}$ of rank $r$ on $X$, the $H$-slope of $\mathcal{G}$ is defined as 
\begin{align*}
\mu_H(\mathcal{G}) := \frac{c_1(\mathcal{G})\cdot H^{n-1}}{r} \in \mathbb{Q}.
\end{align*}
A torsion-free coherent sheaf $\mathcal{G}$ on $X$ is said to be $H$-semistable if $\mu_H(\mathcal{F}) \leq \mu_H(\mathcal{G})$ for 
all subsheaves $\mathcal{F}$ of $\mathcal{G}$.
A vector bundle $E$ on $X$ is called $H$-unstable if it is not $H$-semistable. For every vector bundle $E$ on $X$, there is a unique filtration
\begin{align*}
 0 = E_k \subsetneq E_{k-1} \subsetneq E_{k-2} \subsetneq\cdots\subsetneq E_{1} \subsetneq E_0 = E
\end{align*}
of subbundles of $E$, called the Harder-Narasimhan filtration of $E$, such that $E_i/E_{i+1}$ is $H$-semistable torsion-free sheaf
for each $i \in \{ 0,1,2,\cdots,k-1\}$
and $\mu_H\bigl(E_{k-1}/E_{k}\bigr) > \mu_H\bigl(E_{k-2}/E_{k-1}\bigr) >\cdots> \mu_H\bigl(E_{0}/E_{1}\bigr)$.
We define $Q_1:= E_{0}/E_{1}$ and $\mu_{\min}(E) := \mu_H(Q_1) = \mu_H\bigl(E_{0}/E_{1}\bigr)$. 
We recall the following result from \cite{N99}  or  [Theorem 1.2,\cite{B-B}].
\begin{thm}\label{thm3.1}
 Let $E$ be a vector bundle of rank $r$ on a smooth complex projective variety $X$. Then the following are equivalent
 
 \rm(1) \it $E$ is semistable and $c_2\bigl(\End(E)\bigr) = 0$.
 
 \rm(2) \it $\lambda_E := c_1(\mathcal{O}_{\mathbb{P}(E)}(1)) - \frac{1}{r}\pi^*c_1(E) \in \Nef^1\bigl(\mathbb{P}(E)\bigr)$.
 
 \rm(3) \it For every pair of the form $(\phi,C)$, where $C$ is a smooth projective curve and $\phi : C \longrightarrow X$ is a non-constant morphism, $\phi^*(E)$ is semistable.
\end{thm}
\rm Since nefness of a line bundle does not depend on the fixed polarization (i.e. fixed ample line bundle $H$) on $X$, the Theorem \ref{thm3.1} implies that the semistability of a vector bundle $E$ with $c_2\bigl(\End(E)\bigr) = 0$ is independent of the fixed polarization $H$. We will not mention about the polarization $H$ from now on whenever we speak of semistability of such bundles $E$ with $c_2\bigl(\End(E)\bigr)=0$.
We have the following lemmata as easy applications of Theorem \ref{thm3.1}.
\begin{lem}\label{lem3.2}
 Let $\psi : X \longrightarrow Y$ be a morphism between two smooth complex projective varieties and $E$ is a semistable bundle on $Y$ with $c_2\bigl(\End(E)\bigr) = 0$. Then the pullback bundle $\psi^*(E)$ is also semistable with $c_2\bigl(\End(\psi^*E)\bigr) = 0$.
\begin{proof}
  Let $\phi : C \longrightarrow X$ be a non-constant morphism from a smooth curve $C$ to $X$. If the image of $\phi$ is contained in any fibre of $\psi$, then the pullback bundle $\phi^*\psi^*(E)$ is trivial, and hence semistable. Now let us assume that the image is not contained in any fibre of $\phi$. As $E$ is semistable bundle on $Y$ with $c_2\bigl(\End(E)\bigr) = 0$, by the previous Theorem \ref{thm3.1} the pullback bundle  $(\psi\circ\phi)^*E = \phi^*(\psi^*E)$ under the non-constant morphism $\psi\circ\phi$  is semistable on $C$. Hence $\psi^*E$ is semistable bundle on $X$ with $c_2\bigl(\End(\psi^*E)\bigr) = 0$.
 \end{proof}
\end{lem}

\begin{lem}\label{lem3.3}
 Let  $E$ be a semistable vector bundle of rank $r$ on a smooth complex projective variety $X$ with $c_2\bigl(\End(E)\bigr) = 0$. Then for any positive integer $m$ and any line bundle $\mathcal{L}$ on $X$, $\Sym^m(E)\otimes \mathcal{L}$ is also semistable bundle on $X$ with $c_2\bigl(\End(\Sym^m(E)\otimes \mathcal{L})\bigr) = 0$.
 \begin{proof}
  It is enough to prove that for any smooth complex projective curve $C$ and any non-constant map $\phi : C \longrightarrow X$, the pullback bundle $\phi^*\bigl(\Sym^m(E)\otimes \mathcal{L}\bigr)$ is semistable bundle on $C$. We note that
  \begin{center}
   $\phi^*\bigl(\Sym^m(E)\otimes \mathcal{L}\bigr) = \Sym^m(\phi^*E)\otimes\phi^*(\mathcal{L})$
  \end{center}
As $E$ is a semistable bundle on $X$ with $c_2\bigl(\End(E)\bigr) = 0$, by Theorem \ref{thm3.1} we have $\phi^*E$ is semistable. Hence 
$\Sym^m(\phi^*E)\otimes\phi^*(\mathcal{L})$ is also semistable. This proves the Lemma.
\end{proof}
\end{lem}
The projective bundle $\mathbb{P}_X(E)$ associated to a vector bundle 
$E$ over a projective variety $X$ is defined as $\mathbb{P}_X(E) :=  \Proj\bigl(\bigoplus\limits_{m\geq  0} \Sym^m(E)\bigr)$.
We will simply write $\mathbb{P}(E)$ whenever the base space $X$ is clear from the context. 

\section{Equality of Nef and Pseudo-effective cones} 
We note that for a vector $E$ on a smooth irreducible projective complex variety $X$, the equality
$\Nef^1\bigl(\mathbb{P}(E)\bigr) = \overline{\Eff}^1\bigl(\mathbb{P}(E)\bigr)$ always implies the equality $\Nef^1(X) = \overline{\Eff}^1(X)$. However, 
the converse is not true in general (see section 5, \cite{B-H-P}). In this section, we prove that 
$\Nef^1\bigl(\mathbb{P}(E)\bigr) = \overline{\Eff}^1\bigl(\mathbb{P}(E)\bigr)$ if and only if $\Nef^1(X) = \overline{\Eff}^1(X)$ under the assumption that
$E$ is a semistable bundle on $X$ with $c_2\bigl(\End(E)\bigr) = 0$ (cf Theorem \ref{thm1.2}). We also give several examples, especially on 
abelian varieties and fibre products of projective bundles to illustrate our results.
\begin{proof}[Proof of Theorem \ref{thm1.2}]
\underline{(1)$\Rightarrow$ (2)} Let $D$ be an effective divisor on $\mathbb{P}(E)$ such that \\
$\mathcal{O}_{\mathbb{P}(E)}(D) \simeq \mathcal{O}_{\mathbb{P}(E)}(m) \otimes \pi^*(\mathcal{L})$ for some integer $m$ and a line bundle $\mathcal{L} \in \Pic(X)$. Then 
 \begin{align*}
  H^0\bigl(\mathbb{P}(E), \mathcal{O}_{\mathbb{P}(E)}(m) \otimes \pi^*(\mathcal{L})\bigr) = H^0\bigl(X, \Sym^m(E) \otimes \mathcal{L}\bigr) \neq 0 ,
 \end{align*}
 which implies $m \geq 0$. 
 
Let $\gamma = [C] $ be a movable class in $N_1(X)_{\mathbb{R}}$. Then $C$ belongs to an algebraic family of curves
$\{C_t\}_{t\in T}$ such that  $\bigcup_{t \in T} C_t$ covers a dense subset of $X$. So we can find a curve
$C_{t_1}$ in this family such that  
\begin{align*}
 H^0\bigl( C_{t_1}, \Sym^m(E)\vert_{C_{t_1}} \otimes \mathcal{L}\vert_{C_{t_1}}\bigr) \neq 0.
\end{align*}

Let $\eta_{t_1} :\tilde{C_{t_1}} \longrightarrow C_{t_1}$ be the normalization of the curve $C_{t_1}$ and we call $\phi_{t_1} :=  i\circ \eta_{t_1} $ where $i : C_{t_1} \hookrightarrow X$ is the inclusion.

As $E$ is a semistable bundle on $X$ with $c_2\bigl(\End(E)\bigr) = 0$, $\Sym^m(E) \otimes \mathcal{L}$ is also semistable on $X$ and $c_2\bigl(\Sym^m(E) \otimes \mathcal{L}\bigr) = 0$ by the above Lemma \ref{lem3.3}. Therefore  by Theorem \ref{thm3.1} we have $\phi_{t_1}^*\bigl(\Sym^m(E) \otimes \mathcal{L}\bigr)$ is also semistable on $\tilde{C_{t_1}}$. 

Since $\eta_{t_1}$ is a surjective map, we have
\begin{align*}
 H^0\bigl(\eta_{t_1}^*(\Sym^m(E)\vert_{C_{t_1}} \otimes \mathcal{L}\vert_{C_{t_1}})\bigr)  = H^0\bigl(\phi_{t_1}^*(\Sym^m(E) \otimes \mathcal{L})\bigr)\neq 0.
\end{align*}
This implies that $c_1\bigl(\phi_{t_1}^*\bigl(\Sym^m(E) \otimes \mathcal{L}\bigr)\bigr) \geq 0$ , and hence

$c_1\bigl(\Sym^m(E) \otimes \mathcal{L}\bigr)\cdot C_{t_1} = \bigl\{ c_1\bigl(\Sym^m(E) \otimes \mathcal{L}\bigr) \cdot \gamma \bigr\} \geq 0$ for a movable class $\gamma \in N_1(X)_{\mathbb{R}}$.

Using the duality property of movable cone $\overline{\ME}(X)$ we conclude that
\begin{center}
$ c_1\bigl(\Sym^m(E) \otimes \mathcal{L}\bigr) = c_1\bigl(\pi_*\mathcal{O}_{\mathbb{P}(E)}(D)\bigr) \in \overline{\Eff}^1(X) = \Nef^1(X)$.
\end{center}
\underline{(2) $\Rightarrow$ (3)} We will show that every effective divisor $D$ in $\mathbb{P}(E)$ is nef. Let $\mathcal{O}_{\mathbb{P}(E)}(D) \simeq \mathcal{O}_{\mathbb{P}(E)}(m) \otimes \pi^*(\mathcal{L})$ for some positive integer $m$ and a line bundle $\mathcal{L} \in \Pic(X)$.

 Now 
 $c_1\bigl(\pi_*\mathcal{O}_{\mathbb{P}(E)}(D)\bigr)$
 = $c_1\bigl(\Sym^m(E)\otimes \mathcal{L}\bigr) = \rank\bigl(\Sym^m(E)\bigr)\bigl\{ \frac{m}{r}c_1(E) + c_1(\mathcal{L}) \bigr\} \in \Nef^1(X)$, so that $\frac{m}{r}c_1(E) + c_1(\mathcal{L}) \in \Nef^1(X)$. Since $E$ is a semistable bundle on $X$ with $c_2\bigl(\End(E)\bigr) = 0$, by Theorem \ref{thm3.1} we also have $c_1\bigl(\mathcal{O}_{\mathbb{P}(E)}(1)\bigr) - \frac{1}{r}\pi^*c_1(E) \in \Nef^1\bigl(\mathbb{P}(E)\bigr)$.
 
Hence,
$\mathcal{O}_{\mathbb{P}(E)}(m) \otimes \pi^*(\mathcal{L})$
$\equiv m\bigl(c_1(\mathcal{O}_{\mathbb{P}(E)}(1))\bigr) + \pi^*c_1(\mathcal{L})$

$\equiv m\bigl\{ c_1(\mathcal{O}_{\mathbb{P}(E)}(1)) - \frac{1}{r}\pi^*c_1(E)\bigr\} + \frac{m}{r}\bigl(\pi^*c_1(E) + \pi^*c_1(\mathcal{L})\bigr)$

$\equiv m\bigl\{ c_1(\mathcal{O}_{\mathbb{P}(E)}(1)) - \frac{1}{r}\pi^*c_1(E)\bigr\} + \pi^*\bigl(\frac{m}{r}c_1(E) + c_1(\mathcal{L})\bigr) \in \Nef^1\bigl(\mathbb{P}(E)\bigr)$.

Therefore, $\Eff^1\bigl(\mathbb{P}(E)\bigr) \subset \Nef^1\bigl(\mathbb{P}(E)\bigr) \subset \overline{\Eff}^1\bigl(\mathbb{P}(E)\bigr)$. Taking closure, we get $\Nef^1\bigl(\mathbb{P}(E)\bigr) = \overline{\Eff}^1\bigl(\mathbb{P}(E)\bigr)$.

\underline{(3) $\Rightarrow$ (1)}
We claim that $\Nef^1(X) = \overline{\Eff}^1(X)$. If not then there is an effective divisor $D$ in $X$ which is not nef. Therefore the pullback $\pi^*(D)$ is also effective. Since $\pi$ is a surjective proper morphism, $\pi^*(D)$ can never be nef, which contradicts that $\Nef^1\bigl(\mathbb{P}(E)\bigr) = \overline{\Eff}^1\bigl(\mathbb{P}(E)\bigr).$ Hence we are done.
\vspace{2mm}

In addition, let $\Nef^1(X) = \overline{\Eff}^1(X)$ be a finite polyhedron generated by nef classes $L_1,L_2,\cdots,L_k$. Then for any effective divisor $D$ on $\mathbb{P}(E)$ such that $\mathcal{O}_{\mathbb{P}(E)}(D) \simeq \mathcal{O}_{\mathbb{P}(E)}(m) \otimes \pi^*(\mathcal{L})$ for some positive integer $m$ and $\mathcal{L} \in \Pic(X)$, we have $\frac{m}{r}c_1(E) + c_1(\mathcal{L}) \in \Nef^1(X)$. Let $\frac{m}{r}c_1(E) + c_1(\mathcal{L}) \equiv x_1L_1 + x_2L_2+\cdots+x_kL_k$ (say) for some non-negative real numbers $x_i$'s. The above calculation then shows that 

$\mathcal{O}_{\mathbb{P}(E)}(m) \otimes \pi^*(\mathcal{L}) \equiv m\lambda_E + x_1\pi^*L_1 + x_2\pi^*L_2+\cdots+x_k\pi^*L_k$. 
This shows that
\begin{align*}
 \Eff^1\bigl(\mathbb{P}(E)\bigr) \subseteq \Bigl\{y_0\lambda_E + y_1\pi^*L_1+\cdots+y_k\pi^*L_k\mid y_i \in \mathbb{R}_{\geq 0}\Bigr\} \subseteq \Nef^1\bigl(\mathbb{P}(E)\bigr).
\end{align*}
Therefore
\begin{align*}
\Nef^1\bigl(\mathbb{P}(E)\bigr) = \overline{\Eff}^1\bigl(\mathbb{P}(E)\bigr) = \Bigl\{y_0\lambda_E + y_1\pi^*L_1+\cdots+y_k\pi^*L_k\mid y_i \in \mathbb{R}_{\geq 0}\Bigr\}.
\end{align*}
\end{proof}
\begin{corl}\label{corl3.5}
Let $X$ be a smooth complex projective variety of dimension $n$ with Picard number $\rho(X) = 1$ and $E$ be a semistable vector bundle of rank $r$ on $X$ with $c_2\bigl(\End(E)\bigr) = 0$. Then
\begin{center}
$\Nef^1\bigl(\mathbb{P}(E)\bigr) = \overline{\Eff}^1\bigl(\mathbb{P}(E)\bigr) = \Bigl\{ x_0\lambda_E + x_1(\pi^*L_X) \mid x_0,x_1 \in \mathbb{R}_{\geq 0}\Bigr\}$,
\end{center}
 where $L_X$ is the numerical class of an ample generator of the Ne\'{r}on-Severi group $N^1(X)_{\mathbb{Z}}\simeq \mathbb{Z}$.
 \begin{proof}
 Let $D$ be an effective divisor on $X$ and $D\equiv mL_X$ for some $m\in \mathbb{Z}$. Then $L_X^{n-1}\cdot D = mL_X^{n} \geq 0$. As $L_X$ is ample, by Nakai criterion for ampleness we have $L_X^{n} > 0$. This shows that $m \geq 0$ and $D \equiv mL_X \in \Nef^1(X)$. 
So $\Nef^1(X) = \overline{\Eff}^1(X)$ and hence the result follows from Theorem \ref{thm1.2}.
\end{proof}
\end{corl}
\begin{xrem}\label{remark1}
\rm Miyaoka [Theorem 3.1,\cite{Mi}] proved that a vector bundle $E$ of rank $r$ on a smooth complex projective curve $C$ is semistable if and only if 
\begin{align*}
\Nef^1\bigl(\mathbb{P}(E)\bigr) = \overline{\Eff}^1\bigl(\mathbb{P}(E)\bigr) = \bigl\{ x\lambda_E + y\pi^*\mathcal{L} \mid x,y \in \mathbb{R}_{\geq 0}\bigr\} = \bigl\{ x(\zeta - \mu(E)f) + yf \mid x,y \in \mathbb{R}_{\geq 0}\bigr\},
\end{align*}
where $\zeta = c_1\bigl(\mathcal{O}_{\mathbb{P}(E)}(1)\bigr)$, $f$ is the numerical class of a fibre of $\pi : \mathbb{P}(E)\longrightarrow C$ and $\mathcal{L}$ is an ample generator of $N^1(C)_{\mathbb{Z}} \simeq \mathbb{Z}$. The smooth curve $C$ has Picard number 1 and any second Chern class vanishes on $C$. So the Corollary \ref{corl3.5} can be thought as a partial generalization of Miyaoka's result in higher dimensions. 
\end{xrem}
\begin{exm}\label{corl3.6}
\rm Let $X$ be a smooth projective variety with Picard number $\rho(X) = 1$, and $E = \mathcal{L}_1\oplus \mathcal{L}_2$ be a rank two bundle on $X$ such that $\mathcal{L}_1\cdot L_X = \mathcal{L}_2\cdot L_X$ (Here $L_X$ is the ample generator for $N^1(X)_{\mathbb{Z}}$). Then $E$ is semistable with $c_2\bigl(\End(E)\bigr) = 4c_2(E) - c_1^2(E) = 0$. Therefore, $\Nef^1\bigl(\mathbb{P}(\mathcal{L}_1\oplus \mathcal{L}_2)\bigr) = \overline{\Eff}^1\bigl(\mathbb{P}(\mathcal{L}_1\oplus \mathcal{L}_2)\bigr).$
\end{exm}
\begin{exm}\label{exm3.7}
 \rm Let $X$ be a smooth complex projective variety with Picard number 1 and $q = H^1(X,\mathcal{O}_X) \neq 0$. Then for any line bundle $L$ on $X$, there is a nonsplit extension 
  \begin{align*}
 0 \longrightarrow L \longrightarrow E \longrightarrow L \longrightarrow 0.
 \end{align*}
In this case, $E$ is a semistable bundle of rank 2 with $c_2\bigl(\End(E)\bigr) = 0$. Moreover, for any positive integer $r$, the vector bundles of the forms $E^{\oplus r} \oplus L$ and $E^{\oplus r}$ are examples of semistable bundles 
of ranks $2r+1$ and $2r$ respectively with $c_2\bigl(\End(E^{\oplus r} \oplus L)\bigr) = c_2\bigl(\End(E^{\oplus r})\bigr) = 0$. In all these cases, nef cone and pseudo-effective cones of divisors coincide.
\end{exm}
 \begin{exm}\label{exm3.8}
  \rm Let $G$ be a connected algebraic group acting transitively on a complex projective variety $X$. Then every effective divisor on $X$ is nef, i.e. $\Nef^1(X) = \overline{\Eff}^1(X)$. This is because any irreducible curve $C\subset X$ meets the translate $gD$ of an effective divisor $D$ properly for a general element $g \in G$. Since $G$ is connected, $gD \equiv D$. Therefore $D\cdot C = gD \cdot C \geq 0$, and hence $D$ is nef.  Examples of such homogeneous varieties $X$ include smooth abelian varieties, flag manifolds etc. So for any semistable vector bundle $E$ on such a homogeneous space $X$ with $c_2\bigl(\End(E)\bigr) = 0$, using Theorem \ref{thm1.2} we have $\Nef^1\bigl(\mathbb{P}(E)\bigr) = \overline{\Eff}^1\bigl(\mathbb{P}(E)\bigr)$. For example, let $C$ be a general elliptic curve and $X = C\times C$ be the self product. Then $X$ is an abelian surface and $\Nef^1(X)$ is a non-polyhedral cone (see Lemma 1.5.4 in \cite{L1}). Let $p_i : X \longrightarrow C$ be the projection maps. For any semistable vector bundle $F$ of rank $r$ on $C$, the pullback bundle $E:= p_i^*(F)$ is a semistable bundle with $c_2\bigl(\End(E)\bigr) = 0$. Hence $\Nef^1\bigl(\mathbb{P}(E)\bigr) = \overline{\Eff}^1\bigl(\mathbb{P}(E)\bigr)$.
 \end{exm}
 \begin{exm}
\rm The above example can be extended as follows. Let $B$ be any smooth curve and $C$ be an elliptic curve. Then the product of $C$ with the Jacobian variety of $B$ i.e.,  $C \times J(B)$ is an abelian variety. For any semistable vector bundle $F$ on $C$, the pullback bundle $E := p_1^*(F)$  under the 1st projection $p_1$, is a semistable bundle with $c_2\bigl(\End(E)\bigr) = 0$. Hence $\Nef^1\bigl(\mathbb{P}(E)\bigr) = \overline{\Eff}^1\bigl(\mathbb{P}(E)\bigr)$.
 \end{exm}

 A vector bundle $E$ on an abelian variety $X$ is called weakly-translation invariant (semi-homogeneous in the sense of Mukai) if 
 for every closed point $x \in X$, there is a line bundle $L_x$ on $X$ depending on $x$ such that $T_x^*(E) \simeq E \otimes L_x$ for all $x\in X$, where $T_x$ is the translation morphism given by $x\in X$.
\begin{corl}\label{corl3.8}
 Let $E$ be a semi-homogeneous vector bundle of rank $r$ on an abelian variety $X$. Then $\Nef^1\bigl(\mathbb{P}(E)\bigr) = \overline{\Eff}^1\bigl(\mathbb{P}(E)\bigr)$.
 \begin{proof}
  By a result due to Mukai, $E$ is Gieseker semistable (see Ch1 \cite{HL10} for definition) with respect to some polarization and it has projective Chern classes zero, i.e., if $c(E)$ is the total Chern class, then $c(E) = \Bigl\{ 1+ c_1(E)/r\Bigr\}^r$ (see p. 260, Theorem 5.8 \cite{Muk78}, p. 266, Proposition 6.13 \cite{Muk78}; also see p. 2 \cite{MN84}). Gieseker semistablity implies slope semistability (see \cite{HL10}). So, in particular, we have $E$ is slope semistable with $c_2\bigl(\End(E)\bigr) = 2rc_2(E) - (r -1)c_1^2(E) = 0$. Hence the result follows.
 \end{proof}
\end{corl}
\begin{corl}\label{thm3.9}
 Let $X$ be a smooth complex projective variety $X$ with $\overline{\Eff}^1(X) = \Nef^1(X)$ and $E_1,E_2,\cdots,E_k$ be finitely many semistable vector bundles on $X$ of ranks $r_1,r_2,\cdots,r_k$ respectively with $c_2\bigl(\End(E_i)\bigr) = 0$ for all $i \in \{1,2,\cdots,k\}$. Then
  \begin{center}
 $\Nef^1\bigl(\mathbb{P}(E_1) \times_X \mathbb{P}(E_2)\times_X\cdots\times_X \mathbb{P}(E_k)\bigr) = \overline{\Eff}^1\bigl(\mathbb{P}(E_1) \times_X \mathbb{P}(E_2)\times_X \cdots\times_X \mathbb{P}(E_k)\bigr).$
 \end{center}
  Moreover, if $\Nef^1(X) = \overline{\Eff}^1(X)$ is a finite polyhedron and
\begin{align*}
\pi_r : X_r := \mathbb{P}(E_1) \times_X \mathbb{P}(E_2)\times_X \cdots\times_X \mathbb{P}(E_r) \longrightarrow X_{r-1} := \mathbb{P}(E_1) \times_X \mathbb{P}(E_2)\times_X \cdots\times_X \mathbb{P}(E_{r-1})
\end{align*}
is the projection map for $r\in\{1,2,\cdots,k\}$ with $X_0 = X$, then for each $r$ satisfying $1 \leq r \leq k$
\begin{align*}
\Nef^1\bigl(X_r\bigr) = \overline{\Eff}^1\bigl(X_r\bigr)
= \bigl\{ x_0\lambda_{\psi_{r}^*E_r} + x_1L_1 + x_2L_2 + \cdots+x_rL_r\bigr\}
\end{align*}
where $L_1,L_2,\cdots,L_r$ are the nef generators of the nef cone of $X_{r-1}$ and $\psi_r := \pi_1\circ\pi_2\circ\cdots\circ\pi_{r-1}$.

Conversely, if $X$ is a smooth curve, then the equality $\Nef^1\bigl(\mathbb{P}(E_1) \times_X \cdots\times_X \mathbb{P}(E_k)\bigr) = \overline{\Eff}^1\bigl(\mathbb{P}(E_1) \times_X \cdots\times_X \mathbb{P}(E_k)\bigr)$ implies that $E_i$ is semistable for each $i$.
 \begin{proof}
 We will proceed by induction on $k$. For $k=1$, this is precisely the statement of Theorem \ref{thm1.2}. Now suppose the theorem holds true for $k-1$ many vector bundles. Consider the following fibre product diagram.
 \begin{center}
 \begin{tikzcd}
 X_k:= \mathbb{P}(E_1) \times_X \mathbb{P}(E_2)\times_X \cdots\times_X \mathbb{P}(E_k) \arrow[r, " "] \arrow[d, " \pi_k"]
& \mathbb{P}(E_k)\arrow[d,""]\\
X_{k-1}:= \mathbb{P}(E_1) \times_X \mathbb{P}(E_2)\times_X \cdots\times_X \mathbb{P}(E_{k-1})  \arrow[r, "\psi_k" ]
& X
\end{tikzcd}
\end{center}
Note that 
\begin{center}
$X_k = \mathbb{P}(E_1) \times_X \mathbb{P}(E_2)\times_X \cdots\times_X \mathbb{P}(E_k) \simeq \mathbb{P}_{X_{k-1}}(\psi_k^*E_k)$. 
\end{center}
Since $E_k$ is semistable with $c_2\bigl(\End(E_k)\bigr) = 0$ on $X$, by Lemma \ref{lem3.2} it's pullback $\psi_k^*E_k$ under $\psi_k$ is also semistable with $c_2\bigl(\End(\psi_k^*E_k)\bigr) = 0$. By induction hypothesis we have 
\begin{align*}
 \Nef^1\bigl(\mathbb{P}(E_1) \times_X \mathbb{P}(E_2)\times_X\cdots\times_X \mathbb{P}(E_{k-1})\bigr) = \overline{\Eff}^1\bigl(\mathbb{P}(E_1) \times_X \mathbb{P}(E_2)\times_X \cdots\times_X \mathbb{P}(E_{k-1})\bigr).
\end{align*}
Therefore applying Theorem \ref{thm1.2} we get the result. 

Conversely, if $X$ is a curve and $\Nef^1\bigl(\mathbb{P}(E_1) \times_X  \cdots\times_X\mathbb{P}(E_k)\bigr) = \overline{\Eff}^1\bigl(\mathbb{P}(E_1) \times_X \cdots\times_X \mathbb{P}(E_k)\bigr),$  then inductively 
$\Nef^1\bigl(\mathbb{P}(E_i)\bigr) = \overline{\Eff}^1\bigl(\mathbb{P}(E_i)\bigr)$ for each $i\in\{1,2,\cdots,k\}$. This implies that each $E_i$ is semistable by [Theorem 3.1,\cite{Mi}]. This completes the proof.
 \end{proof}
 \end{corl}
\begin{xrem}
\rm In general, the computations of nef cones $\Nef^1\bigl(\mathbb{P}_X(E)\bigr)$ of any projective bundle $\mathbb{P}_X(E)$ over an arbitrary $X$ can be very complicated, even on smooth surfaces $X$ (see \cite{MN20}). In \cite{K-M}, $\Nef^1\bigl(\mathbb{P}(E_1)\times_C\mathbb{P}(E_2)\bigr)$ of fibre product over a smooth curve $C$ is calculated for any two vector bundles $E_1$ and $E_2$ (not necessarily semistable bundles). We note that the proof in \cite{K-M} easily generalizes to fibre product of finitely many copies of projective bundles over $C$.
\end{xrem}
\begin{xrem}\label{remark2}
 \rm In particular, the above Corollary \ref{thm3.9} generalizes Theorem 4.1 in \cite{K-M-R}. Moreover, Corollary \ref{thm3.9} indicates that the hypothesis about semistability of $E$ in Theorem \ref{thm1.2} is necessary. This is because on a smooth complex projective curve $C$, take one semistable bundle $E_1$ and another unstable bundle $E_2$. Then  $\Nef^1\bigl(\mathbb{P}(E_1) \times_C \mathbb{P}(E_2)\bigl) \subsetneq \overline{\Eff}^1\bigl(\mathbb{P}(E_1) \times_C \mathbb{P}(E_2)\bigr)$. But $\Nef^1\bigl(\mathbb{P}(E_1)\bigr) = \overline{\Eff}^1\bigl(\mathbb{P}(E_1)\bigr)$ and $\mathbb{P}(E_1) \times_C \mathbb{P}(E_2) \simeq \mathbb{P}_{\mathbb{P}(E_1)}(\pi_1^*(E_2))$ where $\pi_1 : \mathbb{P}(E_1) \longrightarrow C$ is the projection map. Note that $\pi_1^*E_2$ is unstable with $c_2\bigl(\End(\pi_1^*E_2)\bigr) = 0$ (see \cite{M20}). In view of this remark, we finish this section with the following question :
 \vspace{2mm}
 
 \it Question : \rm Let $E$ be a vector bundle on a smooth projective variety $X$ and $\Nef^1\bigl(\mathbb{P}_X(E)\bigr) = \overline{\Eff}^1\bigl(\mathbb{P}_X(E)\bigr)$. Does this imply that $E$ is a semistable bundle with $c_2\bigl(\End(E)\bigr) = 0?$
\end{xrem}
\section{Weak Zariski decomposition on fibre product of projective bundles}

Let $E$ and $E'$ be two vector bundles of ranks $m$ and $n$ respectively on a smooth complex projective curve $C$ with $\deg(E) = d$ and $\deg(E') = d'$. Let 
 \begin{align*}
 0 = E_{k} \subset E_{k-1} \subset \cdots\subset E_{1} \subset E_0 = E
 \end{align*}
 and 
 \begin{align*}
 0 = E'_{l} \subset E'_{l-1} \subset \cdots\subset E'_{1} \subset E'_0 = E'
 \end{align*}
be the Harder-Narasimhan filtrations of $E$ and $E'$ respectively. Let us fix $Q_1 := E/E_1 , \rank(Q_1) = r_1 , \deg(Q_1) = d_1 , \mu(Q_1) = \mu_1 = \frac{d_1}{r_1}$ ; $Q'_1 := E'/E'_1 , \rank(Q'_1) = r'_1 , \deg(Q'_1) = d'_1 , \mu(Q'_1) = \mu'_1 = \frac{d'_1}{r'_1}$.
 
Consider the following fibre product diagram:
\begin{center}
 \begin{tikzcd} 
 X = \mathbb{P}(\pi_1^*E) = \mathbb{P}(E) \times_C \mathbb{P}(E') \arrow[r, "\pi'"] \arrow[d, "\pi"]
& \mathbb{P}(E)\arrow[d,"\pi_2"]\\
\mathbb{P}(E') \arrow[r, "\pi_1" ]
& C
\end{tikzcd}
\end{center}
Let $f_1$ and $f_2$ denote the numerical classes of fibres of $\pi_1$ and $\pi_2$ respectively, and $\xi = \pi'^*c_1\bigl(\mathcal{O}_{\mathbb{P}(E)}(1)\bigr)$, \hspace{1mm} $\zeta = \pi^*c_1\bigl(\mathcal{O}_{\mathbb{P}(E')}(1)\bigr)$, \hspace{1mm} $F = \pi^*f_1 = \pi'^*f_2$. Therefore, 
\begin{center}
$N^1(X)_{\mathbb{R}} = N^1\bigl(\mathbb{P}(\pi_1^*(E)\bigr)_{\mathbb{R}} =\bigl\{ x \xi + y \zeta + z F \mid x,y,z \in \mathbb{R} \bigr\}$.
\end{center}
Also, by [Theorem 3.1,\cite{K-M}]
\begin{align*}
 \Nef^1(X) = \Bigl\{ a(\xi - \mu_1F) + b(\zeta - \mu'_1F) + cF \mid a,b,c \in \mathbb{R}_{\geq 0}\Bigr\}.
\end{align*}
We use the above notations in the rest of this section. The intersection products in $X$ are as follows:
\begin{center}
$\xi^mF = \zeta^nF = 0$ , $\xi^{m+1} = \zeta^{n+1} = 0$ , $F^2 = 0$ , $\zeta^n\xi^{m-1} = \deg(E') = d'$ , $\zeta^{n-1}\xi^m = \deg(E) = d$ ,
$\zeta^n = \deg(E')\zeta^{n-1}F = d'\zeta^{n-1}F$ , $\xi^m = \deg(E)\zeta^{m-1}F = d\zeta^{m-1}F$ . 
\end{center}
Since $\pi_1$ is a smooth map, in particular it is flat and hence $\pi_1^*$ is an exact functor. Also we observe that for any semistable vector bundle $V$ on $C$ and for any ample line bundle $H$ on $\mathbb{P}(E')$, $\pi_1^*(V)$ is $H$-semistable with $\mu_H\bigl(\pi_1^*V) = \mu(V)\bigl(f_1\cdot H^{n-1}\bigr)$ and $f_1\cdot H^{n-1} > 0$. These observations immediately imply that 
\begin{align*}
 0 = \pi_1^*E_{k} \subset \pi_1^*E_{k-1} \subset \cdots\subset \pi_1^*E_{1} \subset \pi_1^*E_0 = \pi_1^*E
\end{align*}
is the unique Harder-Narasimhan filtration of $\pi_1^*E$ with respect to any ample line bundle $H$ on $\mathbb{P}(E')$ with $\mathcal{Q}_1 := \pi_1^*E / \pi_1^*E_1 = \pi_1^*Q_1$. Similar argument holds for $\pi_2^*E'$.
We fix $\mathcal{E}_i := \pi_1^*E_i$ for all $i \in \{0,1,2,\cdots,k\}$ and $\mathcal{E}'_i := \pi_2^*E'_i$ for all $i \in \{0,1,2,\cdots,l\}$. Define $\mathcal{E} := \mathcal{E}_0 = \pi_1^*(E)$.

The projection map $p : \mathbb{P}(\mathcal{E}) \backslash \mathbb{P}(\mathcal{Q}_1) \longrightarrow \mathbb{P}(\mathcal{E}_1)$ can be seen as a rational map. The indeterminacies of this rational map are resolved by blowing up $\mathbb{P}(\mathcal{Q}_1)$.
Now we consider the following commutative diagram :
 \begin{center}
 \begin{tikzcd} 
\tilde{X} = \Bl_{\mathbb{P}(\mathcal{Q}_1)}X \arrow[r, "\eta"] \arrow[d, "\beta"]
& \mathbb{P}(\mathcal{E}_1) = Y \arrow[d,"\rho"]\\
X = \mathbb{P}(\mathcal{E}) \arrow[r, "\pi" ]
& \mathbb{P}(E') = \mathcal{M}
\end{tikzcd}
\end{center}
By following the ideas of Proposition 2.4 in \cite{Fu} (See Remark 2.5 \cite{Fu}), one can find 
a locally free sheaf $\mathcal{F}$ on $Y$ such that $\bigl(\tilde{X},Y,\eta\bigr)$ is the projective bundle $\mathbb{P}_Y(\mathcal{F})$ over $Y$ with $\eta$ as the projection map. Moreover, if $\gamma := \mathcal{O}_{\mathbb{P}_Y(\mathcal{F})}(1)$, 
$\xi := \mathcal{O}_{\mathbb{P}(\mathcal{E})}(1)$ and $\xi_1 := \mathcal{O}_{\mathbb{P}(\mathcal{E}_1)}(1)$,
then $\gamma = \beta^*\xi$ and $\eta^*\xi_1 = \beta^*\xi - \tilde{E}$ where $\tilde{E}$ is numerical class of the exceptional divisor of the map $\beta.$
Also, $\tilde{E}\cdot\beta^*\bigl(\xi - \mu_1F\bigr)^{r_1} = 0$.
\vspace{1mm}

In the rest of this section, we give a proof of existence of weak Zariski decomposition for every pseudo-effective $\mathbb{R}$-divisor on
$X = \mathbb{P}(E)\times_C\mathbb{P}(E')$. We consider the following three cases :
\vspace{1mm}

\bf CASE 1 \rm: We assume that $E'$ is semistable and $E$ is unstable with $\rank(Q_1) = m-1$. We show that $\overline{\Eff}^1(X)$ 
is a finite polyhedron and weak Zariski decomposition exists for every extremal ray of  $\overline{\Eff}^1(X)$. In particular, this proves the existence 
of weak Zariski decomposition for every pseudo-effective $\mathbb{R}$-divisor on $X$ in this case (see  Proposition 1 in \cite{M-S-C}).
\vspace{1mm}

\bf CASE 2 \rm: We assume that both $E$ and $E'$ are unstable with  $\rank(Q_1) = m-1$ and  $\rank(Q'_1) = n-1$. The proof of existence of weak Zariski
decomposition in this case is identical to the proof in CASE 1.
\vspace{1mm}

\bf CASE 3 \rm: In this case, we assume that at least one of them, say $E$ is unstable with $\rank(Q_1) < m-1$. We define a cone map 
\begin{center}
$C^{(1)} :  \overline{\Eff}^1\bigl(\mathbb{P}(\mathcal{E}_1)\bigr) \longrightarrow \overline{\Eff}^1\bigl(\mathbb{P}(\mathcal{E})\bigr)$
\end{center}
which is an isomorphism. Inductively, the Harder-Narasimhan filtration of 
$E$ gives the isomorphisms 
\begin{center}
$\overline{\Eff}^1\bigl(\mathbb{P}(\mathcal{E}_{k-1})\bigr) \simeq \cdots \simeq \overline{\Eff}^1\bigl(\mathbb{P}(\mathcal{E}_{1})\bigr) \simeq \overline{\Eff}^1\bigl(\mathbb{P}(\mathcal{E})\bigr)$
\end{center}
of cone maps. We also prove that $C^{(1)}\bigl(D\bigr)$ admits a weak Zariski decomposition if $D$ does. Finally, we show the existence of weak Zariski
decomposition for any pseudo-effective $\mathbb{R}$-divisors on $X$ by using the results in CASE 1 and CASE 2 recursively.

\begin{lem}\label{lem5.3}
 Let $E$ and $E'$ be as above. Further assume that $E'$ is semistable and $E$ is unstable with $\rank(Q_1) = r_1 = m - 1$. Then,
 \begin{align*}
  \overline{\Eff}^1\bigl(\mathbb{P}(E)\times_C\mathbb{P}(E')\bigr) = \Bigl\{a\bigl(\xi + (d_1 - d)F\bigr) + b\bigl(\zeta - \mu'F\bigr) + cF \mid a,b,c \in \mathbb{R}_{\geq 0}\Bigr\},
 \end{align*}
where $\mu' = \mu(E')$. In this case, weak Zariski decomposition exists for pseudo-effective $\mathbb{R}$-divisors in $ X = \mathbb{P}(E)\times_C\mathbb{P}(E')$. 
\begin{proof}
 Note that $\xi + (d_1 - d)F = \bigl[ \mathbb{P}(\mathcal{Q}_1) \bigr] \in \Eff^1(X)$, $\zeta - \mu'F$ , $F \in \Nef^1(X)  \subset \overline{\Eff}^1(X)$. Hence
 \begin{align*}
  \Bigl\{a\bigl(\xi + (d_1 - d)F\bigr) + b\bigl(\zeta - \mu'F\bigr) + cF \mid a,b,c \in \mathbb{R}_{\geq 0}\Bigr\} \subseteq  \overline{\Eff}^1(X).
 \end{align*}
To prove the converse, consider an element $\alpha = a\xi + b\zeta + cF \in  \overline{\Eff}^1(X)$. Then using the intersection products in $X$, we get

$\alpha\cdot F(\zeta - \mu'F)^{n-1}(\xi - \mu_1F)^{m-2} = a \geq 0$ , \hspace{2mm}
$\alpha\cdot F(\zeta - \mu'F)^{n-2}(\xi - \mu_1F)^{m-1} = b \geq 0$,

$\alpha\cdot (\zeta - \mu'F)^{n-1}(\xi - \mu_1F)^{m-1} = \bigl\{ b\mu' - a(d_1-d) + c \bigr\} \geq 0$.

Note that
$\alpha = a\bigl(\xi + (d_1-d)F\bigr) + b\bigl(\zeta - \mu'F\bigr) + \bigl( b\mu' - a(d_1-d) + c ) F $.

This proves the equality. Note that, in this case weak Zariski decompostion exists for every extremal ray of $\overline{\Eff}^1(X)$. 
Hence the result follows from Proposition 1 in \cite{M-S-C}.
\end{proof}
\end{lem}
\begin{lem}\label{lem5.4}
 Let $E$ and $E'$ be two unstable bundles as above with $\rank(Q_1) = r_1 = m - 1$ and $\rank(Q'_1) = r'_1 = n - 1$. Then
 \begin{align*}
  \overline{\Eff}^1\bigl(\mathbb{P}(E)\times_C\mathbb{P}(E')\bigr) = \Bigl\{a\bigl(\xi + (d_1 - d)F\bigr) + b\bigl(\zeta + (d'_1 - d')F\bigr) + cF \mid a,b,c \in \mathbb{R}_{\geq 0}\Bigr\}.
 \end{align*}
 In this case also, weak Zariski decomposition exists for pseudo-effective $\mathbb{R}$-divisors in $ X = \mathbb{P}(E)\times_C\mathbb{P}(E')$. 
\begin{proof}
 The proof is similar to the proof of Lemma \ref{lem5.3}.
\end{proof}
\end{lem}
\begin{lem}\label{lem5.2}
Let $E$ and $E'$ be two bundles as above and $E$ is unstable with $\rank(Q_1) = r_1 < ( m - 1 ).$ Recall the following commutative diagram :
\begin{center}
 \begin{tikzcd} 
\tilde{X} = \mathbb{P}_Y(\mathcal{F}) = \Bl_{\mathbb{P}(\mathcal{Q}_1)}X \arrow[r, "\eta"] \arrow[d, "\beta"]
& \mathbb{P}(\mathcal{E}_1) = Y \arrow[d,"\rho"]\\
X = \mathbb{P}(\mathcal{E}) \arrow[r, "\pi" ]
& \mathbb{P}(E') = \mathcal{M}
\end{tikzcd}
\end{center}
Then

$(i)$ The cone map 
\begin{align*}
 C^{(1)} : \overline{\Eff}^1\bigl(\mathbb{P}(\mathcal{E}_1)\bigr) \longrightarrow \overline{\Eff}^1\bigl(\mathbb{P}(\mathcal{E})\bigr)
\end{align*}
defined as follows :
\begin{align*}
 C^{(1)}(\alpha) := \beta_*\eta^*(\alpha),
\end{align*}
induces an isomorphism of $\overline{\Eff}^1\bigl(\mathbb{P}(\mathcal{E}_1)\bigr)$ onto $\overline{\Eff}^1\bigl(\mathbb{P}(\mathcal{E})\bigr)$.
\vspace{2mm}

$(ii)$ If $D \in \overline{\Eff}^1(Y)$ admits a weak Zariski decomposition, then its image $C^{(1)}(D) \in \overline{\Eff}^1(X)$ under the cone map also 
admits weak Zariski decomposition.
\end{lem}
\begin{proof}
$(i)$ We have $N^1(Y)_{\mathbb{R}} = N^1\bigl(\mathbb{P}(\pi_1^*E_1)\bigr)_{\mathbb{R}} = \{ a\xi_1 + b\nu + c\rho^*f_1 \mid a,b,c \in \mathbb{R}\}$ 
where $\nu = \rho^*c_1(\mathcal{O}_{\mathbb{P}(E')}(1))$. We define $\phi_1 : N^1(X)_{\mathbb{R}} \longrightarrow N^1(Y)_{\mathbb{R}}$ by 
\begin{center}
 $ \phi_1( a\xi + b\zeta + c F) = a\xi_1 + b\nu + c\rho^*f_1$.
\end{center}
This gives an isomorphism between real vector spaces $N^1(X)_{\mathbb{R}}$ and  $N^1(Y)_{\mathbb{R}}$.

Also we define $U_1 : N^1(Y)_{\mathbb{R}} \longrightarrow N^1(X)_{\mathbb{R}}$ as follows :
\begin{center}
$U_1(\alpha) = \beta_*\eta^*(\alpha).$
\end{center}
In particular, $U_1(a\xi_1 + b\nu + c\rho^*f_1) = a\xi + b\zeta + cF$.

As $\eta$ is flat and $\beta$ is bi-rational, the above map $U_1$ is well-defined.

We construct an inverse for $U_1$. Define $D_1 : N^1(X)_{\mathbb{R}} \longrightarrow N^1(Y)_{\mathbb{R}}$ as follows :
\begin{center}
$ D_1(\alpha) = \eta_*(\delta\cdot\beta^*\alpha)$, \hspace{2mm} where $\delta := \beta^*(\xi - \mu_1 F)^{r_1}$. 
\end{center}
Note that 
$D_1(\xi) = \eta_*(\delta\cdot\beta^*\xi) = \eta_*\bigl(\delta\cdot(\eta^*\xi_1 + \tilde{E})\bigr) = \eta_*\bigl(\delta\cdot\eta^*\xi_1 + \delta\cdot \tilde{E}\bigr) = \eta_*\delta\cdot \xi_1 = \bigl[Y\bigr]\cdot \xi_1.$
Similarly, $D_1(\zeta) = \nu$ and $D_1(\pi^*f_1) = \rho^*(f_1)$. This shows that the maps $D_1$ and $\phi_1$ are the same maps.

Next we will show that the map $D_1$ sends effective divisor in $X$ to $\overline{\Eff}^1(Y)$. For any effective divisor $\alpha$, we can write $\beta^*(\alpha) = {\alpha}' + j_*\tilde{\alpha}$ for some $\tilde{\alpha}$, where $\alpha'$ 
is the strict transform under the map $\beta$ and $j : \tilde{E} \longrightarrow \tilde{X}$ is the canonical inclusion. (Here by abuse of notation
$\tilde{E}$ is also the support of $\tilde{E}$). Since $\delta$ is intersection of nef divisors and $\eta_*$ maps pseudo-effective divisors to pseudo-effective 
divisors, it is enough to prove that $\delta \cdot j_*\tilde{\alpha} = 0$. This is true because $\tilde{E}\cdot \delta^{r_1} = 0$ and 
$\tilde{E}\cdot N(\tilde{X}) = j_*N(\tilde{E})$.
\vspace{2mm}

$(ii)$ As $D$ admits weak Zariski decomposition, there is a bi-rational transformation 
$\psi : Y' \rightarrow Y$ such that $\psi^*(D) = P + N$, with $ P \in \Nef^1(Y')$ and $N \in \Eff^1(Y')$. 
Then we have the following commutative diagram :
 \begin{center}
  \begin{tikzcd}
Y' \times_Y \tilde{X} \arrow[r,"\tilde{\psi}"] \arrow[d,"\tilde{\eta}"] & \tilde{X} \arrow[r,"\beta"] \arrow[d,"\eta"] & X \arrow[dl,dashrightarrow,"p"]\\
Y' \arrow[r,"\psi"]
& Y
\end{tikzcd}
\end{center}
We denote $D' := C^{(1)}(D) \in \overline{\Eff}^1(X)$. Let $D' = a\xi + b \zeta + c \pi^*f_1$ for some $a,b,c \in \mathbb{R}$. 
Then by the above construction of the map $C^{(1)}$, we have $D = a\xi_1 + b \nu + c\rho^*f_1$. 

We will show that $\beta^*(D') = \eta^*(D) + a\tilde{E}$. 

Recall that $\nu = \rho^*c_1\bigl(\mathcal{O}_{\mathbb{P}(E')}(1)\bigr)$, $\zeta = \pi^*c_1\bigl(\mathcal{O}_{\mathbb{P}(E')}(1)\bigr)$ and 
$\beta^*\zeta = \eta^*\nu$. Also, $\beta^*\xi = \eta^*\xi_1 + \tilde{E}$. 

Using these relations we conclude

$\beta^*(D')$

$=\beta^*(a\xi+b\zeta+c\pi^*f_1)$

$= a\beta^*\xi + b\beta^*\zeta + c\beta^*(\pi^*f_1)$

$= a\eta^*\xi_1 + a\tilde{E} + b\eta^*\nu + c\eta^*(\rho^*f_1)$

$=\eta^*(a\xi_1+b\nu+c\rho^*f_1) + a\tilde{E}$

$= \eta^*(D) + a\tilde{E}$.
\vspace{2mm}

Therefore, $\bigl(\beta\circ\tilde{\psi}\bigr)^*(D') = \tilde{\psi}^*\beta^*(D')$

$= \psi^*(\eta^*D + a\tilde{E})$

$= \psi^*\eta^*(D) + \psi^*(a\tilde{E})$

$=\tilde{\eta}^*\psi^*(D) + \psi^*(a\tilde{E})$

$=\tilde{\eta}^*(P + N) + \psi^*(a\tilde{E})$

$= \tilde{\eta}^*(P) + \tilde{\eta}^*(N) + \psi^*(a\tilde{E})$.
\vspace{2mm}

The map $\tilde{\psi}$ is bi-rational. This implies $\beta\circ\tilde{\psi}$ is also bi-rational,
$\tilde{\eta}^*(P) \in \Nef^1\bigl(Y' \times_Y \tilde{X} \bigr)$ and $\tilde{\eta}^*(N) + \psi^*(a\tilde{E}) \in \Eff^1\bigl(Y' \times_Y \tilde{X} \bigr)$. Hence we obtain a weak Zariski decomposition of $D' = C^{(1)}(D)$.
\end{proof}

\begin{thm}\label{thm5.4}
 Let $E$ and $E'$ be two vector bundles of rank $m$ and $n$ respectively on a smooth complex projective curve $C$. Then weak Zariski decomposition exists 
 for the fibre product $\mathbb{P}(E)\times_C\mathbb{P}(E')$.
 \begin{proof}
 The following three cases can occur :
 
 \underline{Case I :}
If both $E$ and $E'$ are semistable bundles on $C$, then 
\begin{align*}
\Nef^1\bigl(\mathbb{P}(E)\times_C\mathbb{P}(E')\bigr) = \overline{\Eff}^1\bigl(\mathbb{P}(E)\times_C\mathbb{P}(E')\bigr).
\end{align*}
and hence weak Zariski decomposition exists trivially.

\underline{Case II :}
Let exactly one of $E$ and $E'$ be unstable, and without loss of generality, we assume that $E'$ is semistable. Let
\begin{align*}
 0 = E_{k} \subset E_{k-1} \subset \cdots\subset E_{1} \subset E_0 = E
 \end{align*}
be the Harder-Narasimhan filtration of $E$. If $\rank(E/E_1) = m - 1$, then by Lemma \ref{lem5.3} weak Zariski decomposition exists for any pseudo-effective $\mathbb{R}$-divisors in $\mathbb{P}(E)\times_C\mathbb{P}(E')$. If $\rank(E/E_1) < m - 1$, 
then the cone map as in Lemma \ref{lem5.2}
 \begin{align*}
 C^{(1)} : N^1(Y) \longrightarrow N^1(X)
\end{align*}
induces an isomorphism $C^{(1)} := \beta_*\eta^*\vert_{\overline{\Eff}^1(\mathbb{P}(\pi_1^*E_1))}$ onto $\overline{\Eff}^1(\mathbb{P}(\pi_1^*E))$. 

Recursively, we have isomorphisms of cone maps such that
\begin{align*}
 \overline{\Eff}^1\bigl(\mathbb{P}(\pi_1^*E_{k-1})\bigr) \simeq \cdots \simeq \overline{\Eff}^1\bigl(\mathbb{P}(\pi_1^*E_{1})\bigr) \simeq \overline{\Eff}^1\bigl(\mathbb{P}(\pi_1^*E)\bigr).
\end{align*}
Now the subbundle $E_{k-1}$ is a semistable vector bundle on $C$. Therefore the weak Zariski decomposition exists on $\mathbb{P}(\pi_1^*E_{k-1}) \simeq \mathbb{P}(E_{k-1}) \times_C \mathbb{P}(E')$ by Case I. Since cone maps preserve pseudo-effectivity, by applying Lemma \ref{lem5.2} repeatedly we get that weak Zariski decomposition exists for $\mathbb{P}(\pi_1^*E) \simeq \mathbb{P}(E) \times_C \mathbb{P}(E')$.

\underline{Case III :} 
In this case, both $E$ and $E'$ are unstable. By arguing similarly as in Case II, and by using Lemma \ref{lem5.4} and Case II, we have the existence of weak Zariski decomposition for pseudo-effective $\mathbb{R}$-divisors in this case also.
\end{proof}
\end{thm}
\section{Results about homogeneous bundles}
A vector bundle $E$ on a projective variety $X$ is called $k$-homogeneous if $\overline{\Eff}^k\bigl(\mathbb{P}(E)\bigr) = \Nef^k\bigl(\mathbb{P}(E)\bigr)$. Let $E$ be a vector bundle of rank $r \geq 2$ on a smooth complex projective surface $X$ and $\pi : \mathbb{P}(E) \longrightarrow X$ be the projection map. Then for each 
$k \in \{2,\cdots,r-1\}$,
\begin{align*}
 A^k\bigl(\mathbb{P}(E)\bigr) = \xi^{k} \oplus \xi^{k-1}\pi^*A^1(X) \oplus \xi^{k-2}\pi^*A^2(X),
\end{align*}
where $\xi := c_1\bigl(\mathcal{O}_{\mathbb{P}(E)}(1)\bigr) \in A^1(X)$. We will denote the numerical class of a fibre of $\pi$ by
$F$. We will continue to use the above notations in what follows. Let $L$ be a line bundle on $X$. Then the intersection products on $\mathbb{P}(E)$ are as follows:
\begin{center}
 $\xi^{r-1}F = 1 \hspace{2mm},\hspace{2mm} \xi^{r}\pi^*L= c_1(E)\cdot L \hspace{2mm},\hspace{2mm} \xi^{r-2}(\pi^*L)F = 0 \hspace{2mm},\hspace{2mm} \xi^{r-2}(\pi^*L)(\pi^*L) = L^2$.
\end{center}
We first prove the following useful lemma.
\begin{lem}\label{lem4.1}
 If $E$ is a semistable bundle of rank $r \geq 2$ on a smooth complex projective surface $X$ with $c_2\bigl(\End(E)\bigr) = 0$, then $\lambda_E^{r} = 0$.
 \begin{proof}
 \begin{align}\label{seq0}
 \lambda_E^{r} = \bigl(\xi - \frac{1}{r}\pi^*c_1(E)\bigr)^{r} = \xi^{r} - \xi^{r-1}\pi^*c_1(E) + \frac{1}{r^2}\frac{r(r-1)}{2}\xi^{r-2}(\pi^*c_1(E))^2.
 \end{align}
 Since $c_2\bigl(\End(E)\bigr) = 2rc_2(E) -(r-1)c^2_1(E) = 0$, we have $c^2_1(E) = \frac{2r}{r-1}c_2(E)$. 
 
 \vspace{2mm}
 Replacing this in  (\ref{seq0}) and using Grothendieck's relation we get the result.
\end{proof}
\end{lem}
\begin{thm}\label{thm4.2}
Let $E$ be a semistable vector bundle of rank $r \geq 2$ with $c_2\bigl(\End(E)\bigr) = 0$ on a smooth complex projective surface $X$ of Picard number $\rho(X) = 1$. We denote the numerical class of a fibre of the projection $\pi : \mathbb{P}(E) \longrightarrow X$ by $F$. If $L_X \in N^1(X)_{\mathbb{R}}$ denotes the numerical class  of an ample generator of the N\'{e}ron-Severi group $N^1(X)_{\mathbb{Z}}$, then 
\begin{center}
$\overline{\Eff}^k\bigl(\mathbb{P}(E)\bigr) = \Bigl\{ a \lambda_E^k + b \lambda_E^{k-1}\pi^*L_X + c\lambda_E^{k-2}F \mid a,b,c \in \mathbb{R}_{\geq 0} \Bigr\}$ for all $k$ satisfying $ 1 < k < r$.
\end{center}
Moreover, $E$ is a $k$-homogeneous bundle on $X$ i.e., $\overline{\Eff}^k\bigl(\mathbb{P}(E)\bigr) = \Nef^k\bigl(\mathbb{P}(E)\bigr)$ for all $ 1 \leq k  < r $.
\begin{proof}
By Theorem \ref{thm3.1} $\lambda_E$ is nef. So for any $k\in\{2,\cdots,r-1\}$ $\lambda_E^k$, $\lambda_E^{k-1}\pi^*L_X$ and $\lambda_E^{k -2}F$ are intersections of nef divisors and hence they are pseudo-effective. Also $\lambda_E^r = 0$.

Now if $a\lambda_E^k + b \lambda_E^{k-1}\pi^*L_X + c\lambda_E^{k-2}F \in \overline{\Eff}^k\bigl(\mathbb{P}(E)\bigr)$, then
\vspace{2mm}

$c = \{a\lambda_E^k + b \lambda_E^{k-1}\pi^*L_X + c\lambda_E^{k-2}F\} \cdot \lambda_E^{r+1-k} \geq 0$,
\vspace{2mm}

$bL_X^2 = \{a\lambda_E^k + b \lambda_E^{k-1}\pi^*L_X + c\lambda_E^{k-2}F\} \cdot \lambda_E^{r-k}\pi^*L_X  \geq 0$,
\vspace{2mm} 

This implies $ b \geq 0$ ( Since $L_X$ being ample, $L_X^2 \geq 0$).
\vspace{2mm}

$a = \{a\lambda_E^k + b \lambda_E^{k-1}\pi^*L_X + c\lambda_E^{k-2}F\} \cdot \lambda_E^{r-1-k}F \geq 0$.
\vspace{2mm}

 This proves that for all $k$ satisfying $ 1 < k < r$,
 \begin{align*}
  \overline{\Eff}^k\bigl(\mathbb{P}(E)\bigr) = \Bigl\{ a \lambda_E^k + b \lambda_E^{k-1}\pi^*L_X + c\lambda_E^{k-2}F \mid a,b,c \in \mathbb{R}_{\geq 0} \Bigr\}.
 \end{align*}
 So for any $1 < k < r$,
 \vspace{2mm}
 
 $\overline{\Eff}_k\bigl(\mathbb{P}(E)\bigr) = \overline{\Eff}^{r+1-k}\bigl(\mathbb{P}(E)\bigr)$
 $=\Bigl\{ a \lambda_E^{r+1-k} + b \lambda_E^{r-k}\pi^*L_X + c\lambda_E^{r-1-k}F \mid a,b,c \in \mathbb{R}_{\geq 0} \Bigr\}$.
 \vspace{2mm}
 
 Any element $\alpha := x\lambda_E^k + y \lambda_E^{k-1}\pi^*L_X + z\lambda_E^{k-2}F \in N^k\bigl(\mathbb{P}(E)\bigr)$ with $x,y,z \in \mathbb{R}$ is in  $\Nef^k\bigl(\mathbb{P}(E)\bigr)$ if and only if 
 \vspace{2mm}
 \begin{center}
 $\alpha\cdot\lambda_E^{r+1-k} = z \geq 0$ \hspace{2mm},\hspace{2mm} $\alpha\cdot\lambda_E^{r-k}\pi^*L_X = yL_X^2 \geq 0$ \hspace{2mm} and \hspace{2mm} $\alpha\cdot\lambda_E^{r-1-k}F = x \geq 0$. 
 \end{center}
 Therefore, 
$\overline{\Eff}^k\bigl(\mathbb{P}(E)\bigr) = \Nef^k\bigl(\mathbb{P}(E)\bigr)$ for all $k \in \{1,2,\cdots,r-1\}$.
\end{proof}
\end{thm}
\begin{thm}\label{thm4.3}
Let $\rho : X = \mathbb{P}_C(G) \longrightarrow C$ be a ruled surface defined by a semistable rank 2 bundle $G$ of slope $\mu$ on a smooth complex projective curve $C$ and $\eta := c_1\bigl(\mathcal{O}_{\mathbb{P}_C(G)}(1)\bigr) \in N^1(X)_{\mathbb{R}}$.  Assume that $E$ is a semistable bundle of rank $r \geq 2$ on $X$ with $c_2\bigl(\End(E)\bigr) = 0$. Denote the numerical classes of a fibre of $\rho$ and a fibre of the projection map $\pi : \mathbb{P}(E) \longrightarrow X$ by $f$ and $F$ respectively. Then
\begin{center}
$\overline{\Eff}^k\bigl(\mathbb{P}(E)\bigr) = \Bigl\{ a\lambda_E^k + b\lambda_E^{k-1}\bigl(\pi^*\eta - \mu\pi^*f\bigr) + c\lambda_E^{k-1}\pi^*f + d\lambda_E^{k-2}F \mid a,b,c,d \in \mathbb{R}_{\geq 0}\Bigr\}$,
\end{center}
for all $k$ satisfying $1 < k < r$. Also, 
$E$ is a $k$-homogeneous bundle on $X$ i.e., $\overline{\Eff}^k\bigl(\mathbb{P}(E)\bigr) = \Nef^k\bigl(\mathbb{P}(E)\bigr)$ for all $k \in \{1,2,\cdots,r-1\}$.
\begin{proof}
By Theorem \ref{thm3.1} $\lambda_E$ is nef. Also 
\begin{center}
 $ \Nef^1(X) = \Nef^1\bigl(\mathbb{P}_C(G)\bigr) =\bigl\{ a(\eta - \mu f) + bf \mid a,b \in \mathbb{R}_{\geq 0}\bigr\}$.
\end{center}
  For any $k\in\{2,\cdots,r-1\}$, $\lambda_E^k$, $\lambda_E^{k-1}\bigl(\pi^*\eta - \mu \pi^*f\bigr)$, $\lambda_E^{k-1}\pi^*f$ and $\lambda_E^{k -2}F$ are intersections of nef divisors and hence they are pseudo-effective. 

Now if $a\lambda_E^k + b \lambda_E^{k-1}\pi^*\eta + c\lambda_E^{k-1}\pi^*f + d\lambda_E^{k-2}F \in \overline{\Eff}^k\bigl(\mathbb{P}(E)\bigr)$, then using the intersection products as above we have
\vspace{2mm}

$d = \{a\lambda_E^k + b \lambda_E^{k-1}\pi^*\eta +  c\lambda_E^{k-1}\pi^*f + d\lambda_E^{k-2}F\} \cdot \lambda_E^{r+1-k} \geq 0$,
\vspace{2mm}

$a = \{a\lambda_E^k + b \lambda_E^{k-1}\pi^*\eta + c\lambda_E^{k-1}\pi^*f + d\lambda_E^{k-2}F\} \cdot \lambda_E^{r-1-k}F \geq 0$.
\vspace{2mm}

$b = \{a\lambda_E^k + b \lambda_E^{k-1}\pi^*\eta + c\lambda_E^{k-1}\pi^*f + d\lambda_E^{k-2}F\} \cdot \lambda_E^{r-k}\pi^*f  \geq 0 $.
\vspace{2mm}

$ \{a\lambda_E^k + b \lambda_E^{k-1}\pi^*\eta  + c\lambda_E^{k-1}\pi^*f + d\lambda_E^{k-2}F\} \cdot \lambda_E^{r-k}(\pi^*\eta - \mu \pi^*f) \geq 0$ 
\vspace{2mm}

This implies $ b\eta^2 - \mu b\eta f + c \eta f  = 2b\mu - b\mu + c = \bigl( b\mu + c \bigr) \geq 0$,
\vspace{2mm}

Hence $a\lambda_E^k + b \lambda_E^{k-1}\pi^*\eta + c\lambda_E^{k-1}\pi^*f + d\lambda_E^{k-2}F$
\vspace{2mm}

$= a\lambda_E^k + b\bigl\{\lambda_E^{k-1}(\pi^*\eta - \mu \pi^*f)\bigr\} + (b\mu + c)\lambda_E^{k-1}\pi^*f + d\lambda_E^{k-2}F \hspace{1mm} \in \overline{\Eff}^k\bigl(\mathbb{P}(E)\bigr)$.
\vspace{2mm}

 This proves that 
$\overline{\Eff}^k\bigl(\mathbb{P}(E)\bigr) = \Bigl\{ a\lambda_E^k + b\lambda_E^{k-1}(\pi^*\eta - \mu \pi^*f) + c\lambda_E^{k-1}\pi^*f + d\lambda_E^{k-2}F \mid a,b,c,d \in \mathbb{R}_{\geq 0}\Bigr\}$ for all $k$ satisfying $ 1 < k < r$. A similar argument as in Theorem \ref{thm4.2} will prove that $\overline{\Eff}^k\bigl(\mathbb{P}(E)\bigr) = \Nef^k\bigl(\mathbb{P}(E)\bigr)$ for all $ 0 < k < r$. 
\end{proof}
\end{thm}
\begin{xrem}
\rm Let $E$ be a semistable bundle of rank $r$ on a smooth projective complex variety of dimension $n$ and $c_2\bigl(\End(E)\bigr) = 0$. Then $E$ is projectively flat. So if $c(E)$ is the total Chern class of $E$, then $c(E) = \bigl( 1 + \frac{c_1(E)}{r}\bigr)^r$. In particular, $c_i(E) = \binom ri \bigl(\frac{c_1(E)}{r}\bigr)^i$ for all $i$. This implies that $\lambda_E^r = \bigl(\xi - \frac{1}{r}\pi^*c_1(E)\bigr)^{r} = \sum\limits_i(-1)^i\binom ri \xi^{r-i}\pi^*(\frac{c_1(E)}{r})^i  = 0$ by applying Grothendieck relation. So similar approaches can be taken to calculate $\overline{\Eff}^k(\mathbb{P}(E)), \Nef^k(\mathbb{P}(E))$ over higher dimensional smooth varieties $X$. However the computations of these cones will be complicated in general.
\end{xrem}

\end{document}